\documentclass[CRMATH,Unicode,manuscript]{cedram}
\usepackage[utf8]{inputenc}
\usepackage{comment}
\usepackage{amsmath,amsfonts,amssymb,amsthm}
\usepackage{xcolor}
\usepackage{mathrsfs}
\usepackage{float}
\usepackage{graphicx}
\newtheorem{proposition}{Proposition}[section]
\newtheorem{definition}[proposition]{Definition}
\newtheorem{lemma}[proposition]{Lemma}
\newtheorem{remark}[proposition]{Remark}
\newtheorem{theorem}[proposition]{Theorem}
\newtheorem{corollary}[proposition]{Corollary}
\title{Rough Paths above Weierstrass Functions}

\author{\firstname{Francesco} \lastname{Cellarosi}}
\address{Department of Mathematics and Statistics, Queen's University}
\email[F. Cellarosi]{fc19@queensu.ca}

\author{\firstname{Zachary} \lastname{Selk}}
\address{Department of Mathematics and Statistics, Queen's University}
\email[Z. Selk]{zachary.selk@queensu.ca}


\numberwithin{equation}{section}

\begin{document}

\begin{abstract}
    Rough paths theory allows for a pathwise theory of solutions to differential equations driven by highly irregular signals. The fundamental observation of rough paths theory is that if one can define ``iterated integrals" above a signal, then one can construct solutions to differential equations driven by the signal.

    The typical examples of the signals of interest are stochastic processes such as (fractional) Brownian motion. However, rough paths theory is not inherently random and therefore can treat irregular deterministic driving signals such as a ({vector-valued}) Weierstrass function. 
    This note supplies a construction of a rough path above {vector-valued} Weierstrass functions. 
\end{abstract}

\maketitle

\section{Introduction}
Rough paths theory allows for a pathwise theory of solutions to differential equations driven by highly irregular signals. Rough paths are concerned with differential equations of the form
\begin{equation}\label{eq:diff-eq}
    dY(t)=b(Y(t),t)dt+c(Y(t),t) dX(t),
\end{equation}
where $b,c$ are nice enough functions and $X$ is some ``driving" signal $X(t)=(X^1(t),$ $...,X^d(t))$ with bad regularity - in particular $X$ lacking differentiability. In this case, we may interpret equation \eqref{eq:diff-eq} as the integral equation
\begin{equation}\label{eq:int-eq}
    Y(t)=Y(0)+\int_0^t b(Y(s),s)ds+\int_0^T c(Y(s),s) dX(s).
\end{equation}
However, if $X$ is not of bounded variation the last integral in equation \eqref{eq:int-eq} might not be well-defined as a Riemann-Stieltjes integral. In cases where $X$ is a stochastic process (in particular a martingale), the celebrated It\^o calculus can make sense of the last integral in equation \eqref{eq:int-eq}. It\^o theory averages over paths or takes limits in probability, thus making it unsuitable for pathwise calculus. Rough paths theory takes an alternative approach to the It\^o calculus: by ``enhancing" $X$ with some ``extra information" encoding an integration theory, we can construct pathwise solutions to \eqref{eq:int-eq}.

The fundamental observation of rough paths theory is that the issue of defining solutions to  \eqref{eq:int-eq} can be reduced to defining the iterated integrals $\int_s^t (X^i(r)-X^i(s))dX^j(r)$ for $1\leq i,j\leq d$.  The pair $\left(X(t)-X(s), \int_s^t (X(r)-X(s))\otimes dX(r)\right)$, where $\int_s^t (X(r)-X(s))\otimes dX(r)$ is the matrix of iterated integrals of the form $\int_s^t (X^i(r)-X^i(s))dX^j(r)$, is called a \textbf{rough path} above the signal $X$  {(see definition \ref{def:rough-path})}. 

Rough paths have been extensively applied to stochastic differential equations where $X$ is a stochastic process. Examples of stochastic driving signals include fractional Brownian motion or more general Gaussian processes. However, there is nothing inherently random about rough paths theory. Even in the case of stochastic driving signals, solutions to \eqref{eq:int-eq} are defined pathwise - that is for each realization of the process. Thus even when the driving signal is random, the solution theory can be seen as deterministic. However, in the literature the application of rough paths theory to equations driven by deterministic signals has been limited. This note works out the case when $X$ is a {vector-valued} Weierstrass function.

Given that defining solutions to \eqref{eq:int-eq} can be reduced to defining a rough path (iterated integral) above the driving signal, there has been considerable interest in the construction of rough paths above various signals. If the driving signal is {scalar-valued}, then we can make the simple definition that $\int_s^t (X(r)-X(s))dX(r):=\frac{1}{2}(X(t)-X(s))^2$. This is a completely acceptable definition for any $X$, so the primary interest is in defining rough paths above {vector-valued} signals.

In \cite{Lyons-Victoir}, the authors show that any signal that is $\alpha$-H\"older (or has finite $p$-variation) has a rough path lift. However, their construction is fairly abstract and hence there is a lot of interest in more natural constructions - in particular from natural approximations. {A more natural construction for general $\alpha$-H\"older paths can be found in \cite{Zambotti-Sewing}.} The construction of a rough path above fractional Brownian motion using Volterra's representation is given in \cite{tindel-fbm}. In \cite{Jain-Monrad-Condition}, the authors constructed rough paths above very general Gaussian signals. 

Weierstrass functions are real-valued functions which have been identified as the ``deterministic analogue" of Brownian motion (this is the perspective taken in \cite{SLE-Weierstrass}, for example). In particular, Weierstrass functions have similar analytic properties to the types of stochastic processes that rough paths theory was invented to handle. The analytic properties of Weierstrass functions have been studied extensively. For example, see \cite{MR3255455,MR2896736,MR2952868,MR1452806,MR3606737,MR3803788} for discussion of the Hausdorff dimension of the graphs of Weierstrass functions. Another interesting instance in which a Weierstrass function is used in place of a stochastic process is given in \cite{SLE-Weierstrass}, where a deterministic version of the Schramm–Loewner evolution is considered, driven by a Weistrass function rather than by a Brownian motion.

We emphasize that rough paths theory for {scalar-valued} functions is somewhat trivial in the sense mentioned above. Therefore, we will consider a {vector-valued} Weierstrass function. By this, we mean an $\mathbb R^d$-valued function whose components are {scalar-valued} Weierstrass functions with possibly different parameters. A two dimensional Weierstrass function has been studied in \cite{Imkeller-Weierstrass,Imkeller-Weierstrass-2}. In \cite{Imkeller-Weierstrass}, {Example 2.8}, the authors consider the two dimensional Weierstrass function
$$W: t\mapsto \left(\sum_{k=0}^\infty a_k \cos(2^k \pi t), \sum_{k=0}^\infty a_k \sin(2^k \pi t))\right),$$ 
where they show that the iterated integrals over $(-1,1)$ of the approximating finite sums diverge (therefore the most natural approximation scheme to $W$ will not lead to a rough path above $W$). Furthermore, in \cite{Imkeller-Weierstrass-2} they discuss the smoothness of the SBR measure. 

We consider a similar signal but of any dimension and with only {cosines} (or only {sines}) instead of mixed {cosines} and {sines}. We construct a rough path above a {vector-valued} Weierstrass function without any of the probabilistic technicalities generally present in rough paths theory to demonstrate the power of rough paths theory beyond the setting of stochastic differential equations. 

The main result is the following theorem. See Theorem \ref{theorem:main-full} for a full version.

\begin{theorem}[Main theorem, simplified version]\label{theorem:main}
    Let $0<a_1,...,a_d<1$ and let $b_1,...,b_d\geq 2$ be integers with $a_1b_1>1,..., a_db_d>1$. Let $$W(t)=(W^{1}(t),\ldots,W^d(t))=\left(\sum_{k=0}^\infty a_1^k \cos( b_1^k \pi t),...,\sum_{k=0}^\infty a_d^k \cos(b_d^k \pi t)\right)$$ and let $$W_N(t)=(W^{1}_N(t),\ldots,W^d_N(t))=\left(\sum_{k=0}^N a_1^k \cos( b_1^k \pi t),...,\sum_{k=0}^N a_d^k \cos(b_d^k \pi t)\right).$$ Then for $1\leq i,j\leq d$ and $0\leq s\leq t\leq 1$ the limit $$I^{i,j}(s,t):=\lim_{N\to\infty} \int_s^t (W^{i}_N(r)-W^{i}_N(s))\,dW^{j}_N(r)$$ exists. Furthermore, if we assume that $-\frac{\ln{a_i}}{\ln{b_i}}>1/3$ for $1\leq i\leq d$, then the matrix with entries $I^{i,j}(s,t)$ defines a geometric rough path above $W$.
    
\end{theorem}

Again, we emphasize that we construct the rough path above a {vector-valued} Weierstrass function because for a {scalar-valued} Weierstrass function we can always define $\int_s^t (W(r)-W(s))dW(r):=\frac{1}{2}(W(t)-W(s))^2$. The technical challenge is defining $\int_s^t (W^i(r)-W^i(s))dW^j(r)$ where $W^i$ and $W^j$ are two different one-dimensional Weierstrass functions. For an example of {vector-valued} Weierstrass functions, see Figure \ref{fig:two-Weierstrass-functions}.

{\section*{Acknowledgements}We would like to greatly thank the anonymous referees who in detail gave constructive feedback including new results such as Propositions \ref{proposition:ref-1}, \ref{prop:ref-2} and a strengthening of our main Theorem \ref{theorem:main-full} to include geometric rate of convergence.}

\begin{figure}
    \centering
\includegraphics[width=11cm]{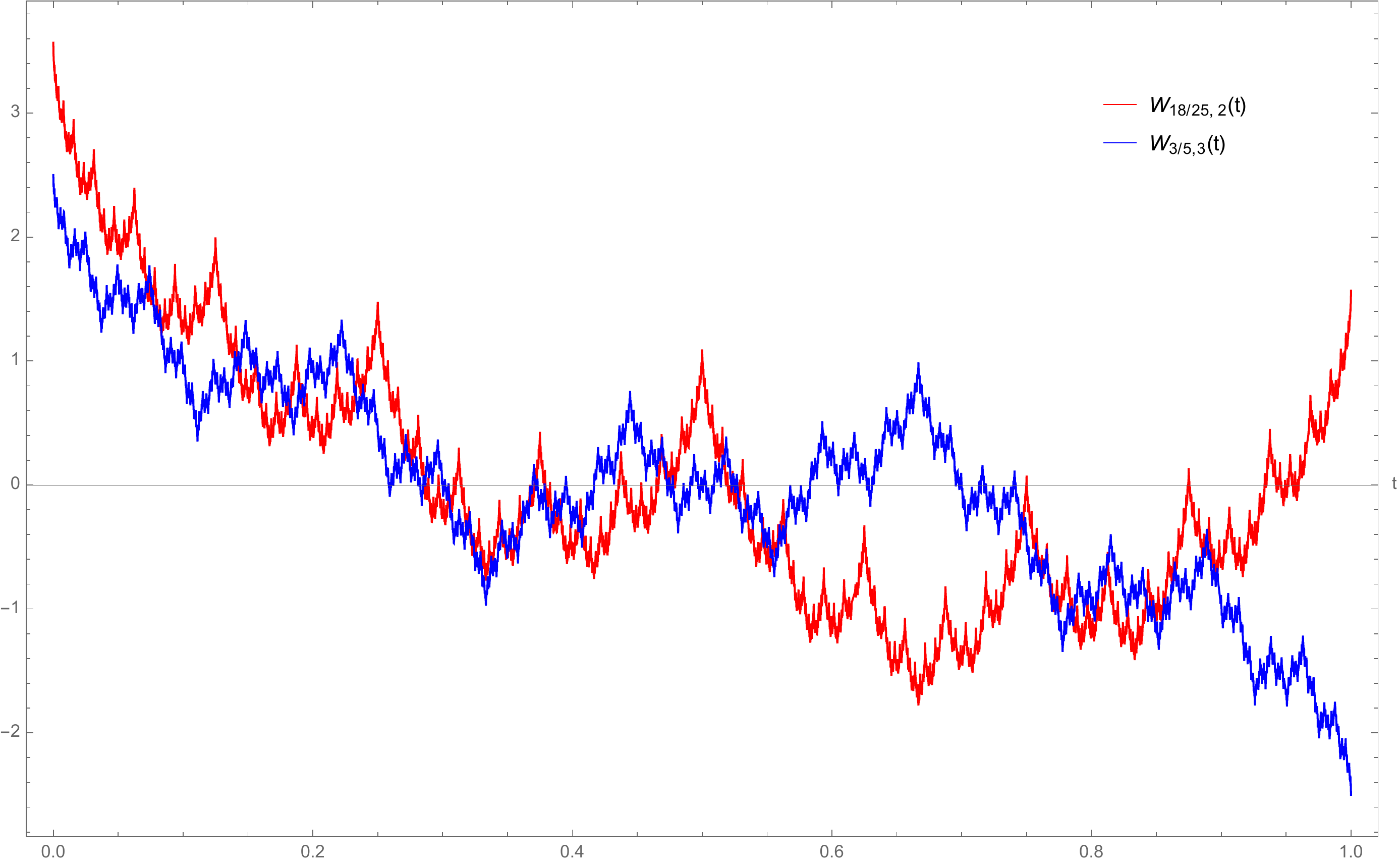}
\includegraphics[width=11cm]{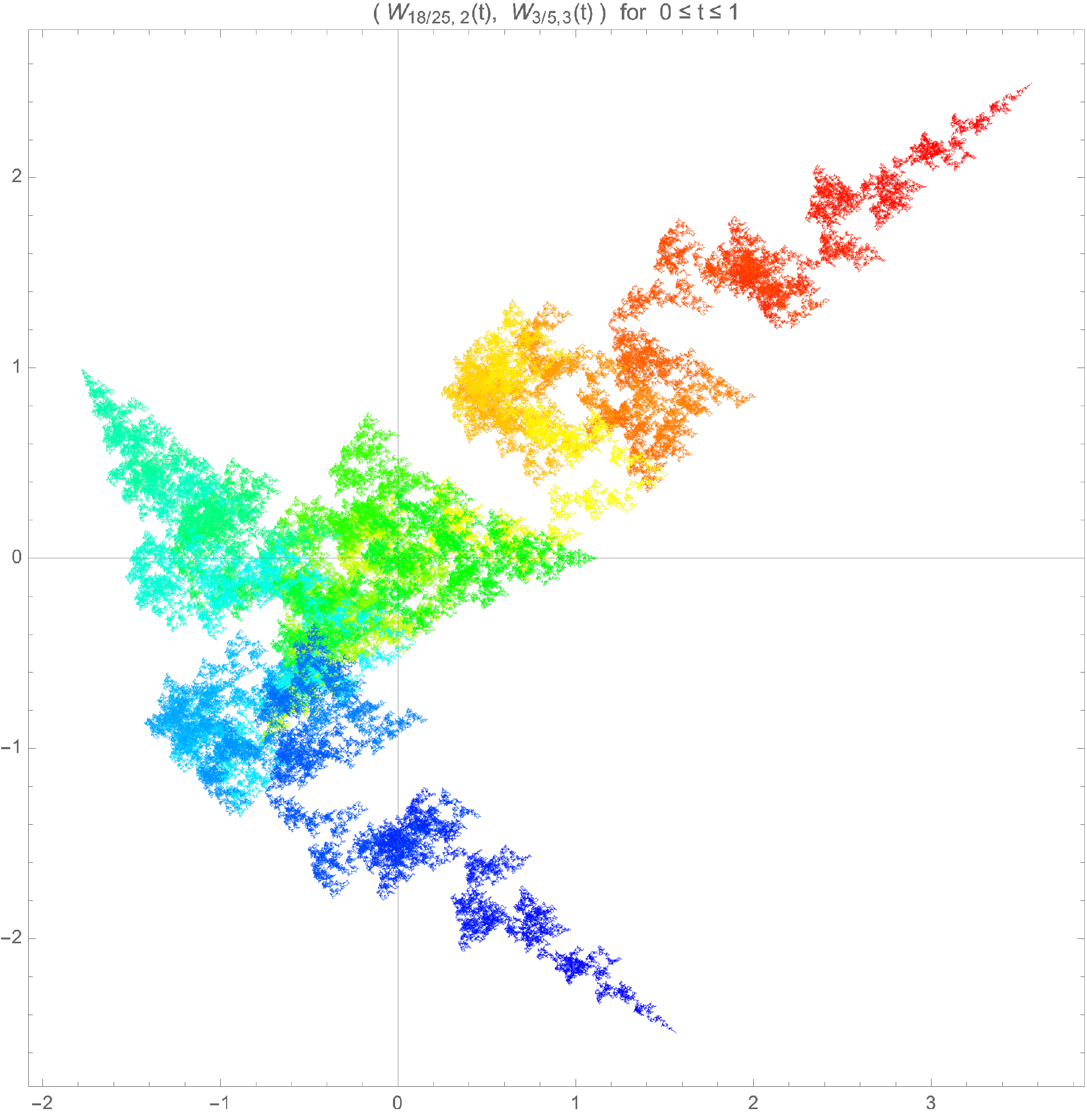}
    \caption{Top panel: two Weierstrass functions with parameters $b_1=2, a_1=\frac{18}{25}$ (in red) and $b_2=3, a_2=\frac{3}{5}$ (in blue) so that $\alpha_1\approx0.473931$ and $\alpha_2\approx0.464974$ are in the range prescribed by Theorems \ref{theorem:main} and \ref{theorem:main-full}. Bottom panel: the curve $[0,1]\to\mathbb{R}^2$ given by $t\mapsto (W_{18/25,2}(t), W_{3/5,3}(t))$.
    Note that on the interval $[1,2]$ the path is the same as for $[0,1]$ but run backwards.
    }
    \label{fig:two-Weierstrass-functions}
\end{figure}

\section{Background on Rough Paths}
In this section, we give a brief introduction to rough paths theory. See \cite{Friz-Hairer-Book,Friz-Victoir-Book} for two excellent introductions to rough paths.

Rough paths theory treats H\"older continuous driving signals - or in the case of stochastic differential equations, stochastic processes that are almost surely H\"older. We recall the definition of H\"older continuity below.

\begin{definition}\label{def:Holder}
    Let $\alpha\in (0,1]$. We denote by   $C^\alpha([0,T],\mathbb R^d)$ the space of $\alpha$-H\"older functions $f:[0,T]\to \mathbb R^d$ equipped with the norm $\|f\|_\alpha:=\sup_{t\neq s}\frac{|f(t)-f(s)|}{|t-s|^\alpha}.$
\end{definition}
In order to solve equation \eqref{eq:int-eq} through some type of Picard iteration, one must first make sense of integrals of the form 
\begin{equation}\label{eq:int-picard}
    \int_0^t c(X(s),s) dX(s).
\end{equation}
The following result by  Young is classical. 
\begin{theorem}[{Young}  \cite{Young-1936}]
    Let $f\in C^\alpha{([0,T],\mathbb R^d)}$ and let $g\in C^\beta{([0,T],\mathbb R^d)}$ with $\alpha+\beta>1$. Let $\mathcal P_n=\{0=t_0<t_1<...<t_n=T\}$ be a sequence of partitions whose mesh size tends to $0$. Then the limit
    \begin{equation}\label{eq:Young-integral}
        \int_0^T f(s) dg(s)=\lim_{n\to \infty}\sum_{k=0}^n f(t_k)(g(t_{k+1})-g(t_k))
    \end{equation}
    exists and is independent of the sequence of partitions. The integral in equation \eqref{eq:Young-integral} is known as the Young or Riemann-Stieltjes integral.
\end{theorem}
Therefore if we assume that $c$ is smooth, then the integral in equation \eqref{eq:int-picard} makes sense as a Young integral if $X\in C^\alpha{([0,T],\mathbb R^d)}$ for $\alpha>1/2$. However, for lower regularity driving signals, the Riemann-Stieltjes integral defined in equation \eqref{eq:Young-integral} needs to include extra terms which represent the ``iterated integrals" of $X$ against itself. We provide a heuristic derivation below.

\textbf{Heuristic Derivation of Rough Integral}\textit{ If $X$ is a signal that is $\alpha$-H\"older continuous with $\alpha\in (1/3,1/2]$ and $F$ is a smooth function, then for a partition of $[0,t]$, $\mathcal P=\{0=t_0<...<t_n=t\}$ we have that (the a-priori ill defined) integral
\begin{align*}
    &\int_0^t F(X(r)) dX(r)=\sum_{k=0}^n \int_{t_k}^{t_{k+1}} F(X(r))dX(r)\\
    &=\sum_{k=0}^n \int_{t_k}^{t_{k+1}} {\bigg(}F(X(t_k))+F'(X(t_k))(X(r)-X(t_k))+O(|r-t_k|^{2\alpha}){\bigg)}dX(r)\\
    &=\sum_{k=0}^n \bigg(F(X(t_k))(X(t_{k+1})-X(t_k))+F'(X(t_k))\int_{t_k}^{t_{k+1}}(X(r)-X(t_k)) dX(r)\\
&\qquad\quad+O(|t_{k+1}-t_k|^{3\alpha})\bigg).
\end{align*}
As $3\alpha>1$, the remainder term should tend to $0$ as the mesh size of the partition $|\mathcal P|=\displaystyle\max_{0\leq k\leq n-1}(t_{k+1}-t_k)$ tends to $0$. \\\\
Note that we restricted $\alpha\in (1/3,1/2]$. This is purely for exposition - for lower $\alpha$ we would need more iterated integrals. In particular, if $\alpha\in (\frac{1}{n+1},\frac{1}{n}]$ then we need $n-1$ additional terms so that $\alpha(n+1)>1$. In the rest of this note, we will stick to $\alpha\in (1/3,1/2]$ for convenience.} 

This heuristic reduces the problem of defining the integral $\int_0^t F(X(r))dX(r)$ to just defining $\int_{t_k}^{t_{k+1}}(X(r)-X(t_k)) dX(r)$. However, if $X$ is irregular then the iterated integral does not exist ``canonically" as a Riemann-Stieltjes integral and therefore must be \textit{postulated}. There are two main properties we would like such a definition to satisfy - one analytic and one algebraic.

A rough path above a signal $X\in C^\alpha([0,T],\mathbb R^d)$ is therefore a pair $\mathbf{X}_{s,t}=(X_{s,t}, \mathbb X_{s,t})$ where $X_{s,t}=X(t)-X(s)$ is the increment of $X$ and $\mathbb X_{s,t}$ is a \textit{definition} or \textit{postulation} of the iterated integral $\int_{s}^{t}(X(r)-X(s)) \otimes dX(r)$ where again we note that $\int_{s}^{t}(X(r)-X(s)) \otimes dX(r)$ represents the matrix whose $(i,j)$th entry is $\int_{s}^{t}(X^i(r)-X^i(s))  dX^j(r)$. We give a precise definition below.

\begin{definition}\label{def:rough-path}
    Let $T>0$ and let $\Delta_2^{(0,T)}=\{(s,t):0\leq s\leq t\leq T\}$ denote the $2$-simplex. Given a signal $X\in C^\alpha([0,T], \mathbb R^d)$ with $\alpha\in (1/3,1/2]$ we say $\mathbf X=(X, \mathbb X):\Delta_2^{(0,T)}\to \mathbb R^d\oplus \mathbb R^{d\times d}$ is a \textbf{rough path} above $X$ if for all $s,u,t\in [0,T]$ we have
    \begin{align*}
    &(i) \qquad X_{s,t}=X(t)-X(s)\\
        &(ii) \qquad\mathbb X_{s,t}-\mathbb X_{s,u}-\mathbb X_{u,t}=X_{s,u}\otimes X_{u,t}\\
        &(iii)\qquad \|\mathbf X\|_{\alpha, 2\alpha}:=\sup_{t\neq s}\frac{|X_{s,t}|}{|t-s|^\alpha}+\sup_{t\neq s}\frac{|\mathbb X_{s,t}|}{|t-s|^{2\alpha}}<+\infty,
    \end{align*}
where $|\cdot|$ denotes any of the (equivalent) norms on either $\mathbb R^d$ or $\mathbb R^{d\times d}.$ We denote the space of rough paths by $\mathscr C^{\alpha}$. The topology generated by the seminorm $\|\cdot\|_{\alpha, 2\alpha}$ is called the \textbf{rough topology}.
\end{definition}

\begin{remark}
    It should be stressed that in Definition \ref{def:rough-path} there is no reference to any iterated integrals. The second order process $\mathbb X$ intuitively encodes a notion of iterated integral but formally, $\mathbb X$ is just some process that takes values in $\mathbb R^{d\times d}$. Once we have this purely abstract object $\mathbb X$ we may make the definition
    \begin{equation}
        \int_s^t (X(r)-X(s))\otimes dX(r) := \mathbb X_{s,t}.
    \end{equation}
\end{remark}

\begin{remark}\label{remark:Levy-area}
    The difference of iterated integrals
    $$L^{i,j}(s,t):=\int_s^t X^i(r) dX^j(r)-\int_s^t X^j(r) dX^i(r)$$
    is often referred to as the \textbf{L\'evy area}. 
    The reason for this terminology is that if $X^i$ and $X^j$ were piecewise smooth, then by Green-Stokes theorem, the (signed) area swept by the curve $(X^i, X^j)$ in $\mathbb R^2$ from time $s$ to time $t$ would be twice $L^{i,j}(s,t)$.
\end{remark}

One can check that if $X\in C^\infty$ and $\int_s^t (X(r)-X(s)) dX(r)$ is the Riemann-Stieltjes integral, then $\mathbf{X}_{s,t}:=\left(X(t)-X(s),\int_{s}^{t}(X(r)-X(s)) dX(r)\right)$ satisfies Definition \ref{def:rough-path}. In this sense, there is a natural embedding of $C^{\infty}$ smooth functions into the space of rough paths. The closure of these functions under the rough topology is the space of geometric rough paths.

\begin{definition}
    For $\alpha\in (1/3,1/2]$, denote by $\mathring{\mathscr C}_g^{\alpha}$ the image of the embedding $\iota: C^\infty([0,T],\mathbb R^d)\to\mathscr C^\alpha$ where $\iota(f)_{s,t}=\left(f(t)-f(s),\int_s^t (f(r)-f(s))f'(r) dr\right)$. Let $\mathscr C_g^\alpha$ be the closure of $\mathring{\mathscr C}_g^{\alpha}$ under the seminorm $\|\cdot \|_{\alpha, 2\alpha}$ defined in Definition \ref{def:rough-path}. We call $\mathscr C_g^\alpha$ the space of $\textbf{geometric rough paths}$. 
\end{definition}

\begin{remark}
    The space of rough paths $\mathscr C^\alpha$, although a subset of the vector space $\mathbb R^d \oplus \mathbb R^{d\times d}$ is not itself a vector space because of the nonlinear relation (ii) in Definition \ref{def:rough-path}. However, the seminorm $\|\cdot\|_{\alpha,2\alpha}$ can be used to define a metric - thus $\mathscr C^\alpha$ is a metric space. 
\end{remark}

We can transfer pointwise convergent sequence of approximations to a rough path to convergence in the rough topology using the following ``interpolation" theorem.

\begin{theorem}\label{theorem:interpolation}(see \cite{Friz-Hairer-Book}, Exercise 2.9)
    Let $(X_N, \mathbb X_N)\in \mathscr C^\beta$ for some $\beta\in (1/3,1,2]$ and every $N\geq1$. Suppose that the following uniform bounds hold
    {\begin{align*}
        \sup_N \|X_N\|_\beta &<\infty,\\
        \sup_N \|\mathbb X_N\|_{2\beta} &<\infty.
    \end{align*}}
    Suppose that, as $N\to\infty$,  $X_N(t)\to X(t)$ 
    and $\mathbb X_N(0,t)\to \mathbb X(0,t)$ pointwise for all $t$. Then $(X,\mathbb X)\in \mathscr C^\beta$ and $(X_N,\mathbb X_N)$ converges to $(X, \mathbb X)$ in $\mathscr C_g^\alpha$ as $N\to\infty$ for all $1/3<\alpha<\beta$.
\end{theorem}

\section{Weierstrass Function}

Let $b\geq 2$ be an integer and  $0<a<1$ with $ab>1$. 
Define the ({scalar-valued}) Weierstrass function by 
\begin{equation}
    W_{a,b}(t):=\sum_{n=0}^\infty a^{n} \cos(b^n \pi t).\label{def-W-a-b}
\end{equation}
In 1872 Weierstrass proved that $W_{a,b}$ is nowhere differentiable when $b$ is odd and $ab>1+\frac{3\pi}{2}$. The range of parameters $a,b$ for which the nondifferentiability holds was enlarged by several authors, culminating in the work of Hardy \cite{Hardy}. For the parameters we are considering, Hardy proved that $W_{a,b}$ is nowhere differentiable and $\alpha$-H\"{o}lder continuous with $\alpha=-\frac{\ln{a}}{\ln{b}}$ (\cite{Hardy}, Theorems 1.31 and 1.32). That is, according to Definition \ref{def:Holder},  $W_{a,b}\in C^{\alpha}([0,T],\mathbb R)$ for every $T>0$.
With $\alpha$ as above, we write the equivalent representation
\begin{equation}
    W_{b,\alpha }(t):=\sum_{n=0}^\infty b^{-n\alpha} \cos(b^n \pi t).\label{def-W-b-alpha}
\end{equation}
Also define its truncated version
\begin{equation}
    W_{b,\alpha,N}(t):=\sum_{n=0}^N b^{-n\alpha}\cos(b^n \pi t).
\end{equation}

Even though it is well known that the Weierstrass function is H\"older continuous, for expository reasons we will include a short proof here.

In the proof of Proposition \ref{prop:Holder-continuity} and Proposition \ref{prop:uniform-bound-Levy-area}, we will write ``$\lesssim$'' to denote ``$\leq K $'' where $K$ is some unspecified absolute constant whose value may change from line to line. This is equivalent to Vinogradov's ``$\ll$'' notation.

For simplicity, instead of working on an arbitrary interval $[0,T]$, we henceforth set $T=1$.

\begin{proposition}\label{prop:Holder-continuity}
    Let {$W_{b,\alpha}$} be a Weierstrass function. Then {$W_{b,\alpha}\in C^\alpha([0,1],\mathbb R)$} for $\alpha=-\log_b(a)$.
\end{proposition}
\begin{proof}
Consider some pair $(s,t)$ and choose $N$ so that $b^{-(N+1)}<|t-s|\leq b^{-N}$. Then we have that
\begin{align*}
    &|W_{b,\alpha}(t)-W_{b,\alpha}(s)|\leq \sum_{n=0}^\infty b^{-n\alpha} |\cos(b^n\pi t)-\cos(b^n\pi s)|\\
    &= \sum_{n=0}^N b^{-n\alpha} |\cos(b^n\pi t)-\cos(b^n\pi s)|+\sum_{n=N+1}^\infty b^{-n\alpha} |\cos(b^n\pi t)-\cos(b^n\pi s)|\\
    &\leq \sum_{n=0}^N b^{-n\alpha} b^n \pi |t-s|+\sum_{n=N+1}^\infty b^{-n\alpha} 2\\
    &\lesssim b^{(1-\alpha)(N+1)}b^{-N}+b^{-\alpha N}\\
    &\lesssim b^{-\alpha N}\\
    &\lesssim |t-s|^\alpha.
\end{align*}
\end{proof}
The primary technical barrier to constructing a rough path above a {vector-valued} {function} is the uniform bound on the H\"older norm of the iterated integrals of the approximating sequence. Thus, the following Proposition \ref{prop:uniform-bound-Levy-area} is the primary technical contribution of this paper. We use a similar idea of the proof of Proposition \ref{prop:Holder-continuity} of splitting the sum into a bulk and a tail. However, as there is a double sum we will need to apply two cutoffs and four separate bounds. The bound involving the tail of each sum is the most interesting, and involves various cases depending on whether $b_1, b_2$ are multiplicatively independent or not, see bound \textbf{(iv)} below. 

\begin{proposition}\label{prop:uniform-bound-Levy-area} Let $W_{b_1,\alpha_1}$ and $W_{b_2,\alpha_2}$ be two Weierstrass functions with truncated versions $W_{b_1,\alpha_1,N}$ and $W_{b_2,\alpha_2,N}$. Let $\varepsilon>0$ be such that $\varepsilon<\min\{\alpha_1,\alpha_2\}$. 
Define the approximating sequence of iterated integrals as 
\begin{equation}\label{I^N_(b1,alpha1)(b2,alpha2)}
  I^N(s,t)=I_{(b_1,\alpha_1),(b_2,\alpha_2)}^N(s,t):=\int_s^t (W_{b_1,\alpha_1,N}(r)-W_{b_1,\alpha_1,N}(s))\, d W_{b_2,\alpha_2,N}(r).
\end{equation}
Then there exists some constant $C:=C(\alpha_1,\alpha_2,b_1,b_2,\varepsilon)$, independent of $N$ so that
\begin{equation}\label{eq:bound-on-I-1}
    \left| I^N(s,t)\right|\leq C |t-s|^{\alpha_1+\alpha_2-2\varepsilon},
\end{equation}
for all $0\leq s\leq t\leq 1.$
\end{proposition}
We will need the following result due to N.I. Fel'dman.
\begin{lemma}[\cite{Bugeaud}, Theorem 1.8]\label{lem:Feldman}
Let $b_1, b_2\geq 2$ be two integers that are multiplicatively independent (i.e. $\ln{(b_1)}/\ln{(b_2)}$ is irrational). Let $n,\ell\geq1$  be two integers. Then there exists a constant $P>0$, depending only on $b_1,b_2$ such that
$$|b_1^n-b_2^\ell|\geq \frac{\max\{b_1^n,b_2^\ell\}}{(\max\{3,n,\ell\})^P}.$$
\end{lemma} 
    
\begin{proof}[Proof of Proposition \ref{prop:uniform-bound-Levy-area}]
Recalling that $\alpha_i=-\frac{\ln(a_i)}{\ln(b_i)}$ for $i=1,2$, we may rewrite $I^N(s,t)$ as 
\begin{equation}
    I^N(s,t)=\sum_{n=0}^N\sum_{\ell=0}^N a_1^{n} a_2^{\ell}\, J^{n,\ell}(s,t),
\end{equation}
where
\begin{equation}
    J^{n,\ell}(s,t):= \int_s^t \left(\cos(b_1^n \pi r)-\cos(b_1^n \pi s)\right) d\cos(b_2^\ell \pi r).
\end{equation}

We will need four bounds for $J$. 

\textbf{Bound (i)}.
Differentiating the integrator of $J^{n,\ell}(s,t)$ gives that
\begin{align*}
    |J^{n,\ell}(s,t)|&=\left|\int_s^t \left(\cos(b_1^n \pi r)-\cos(b_1^n \pi s)\right)  \pi b_2^\ell \sin(b_2^\ell \pi r)\, dr\right|\\
    &\leq \int_s^t \left|\left(\cos(b_1^n \pi r)-\cos(b_1^n \pi s)\right)  \pi b_2^\ell\sin(b_2^\ell \pi r)\right| dr\\
    &{\lesssim  b_1^n b_2^\ell \int_s^t (r-s)\, dr}\\
    &\lesssim b_1^n b_2^\ell (t-s)^2,
\end{align*}
where the second to last inequality is due to the Lipschitz property of {cosine} and the trivial bound $|\sin(x)|\leq 1$.

\textbf{Bound (ii)}.
Using the trivial bounds $|\sin(x)|,|\cos(x)|\leq 1$ we have
\begin{align*}
     |J^{n,\ell}(s,t)| &\leq \int_s^t \left|\left(\cos(b_1^n \pi r)-\cos(b_1^n \pi s)\right)  \pi b_2^\ell\sin(b_2^\ell \pi r)\right| dr\\
    &\lesssim b_2^\ell (t-s).
\end{align*}

\textbf{Bound (iii)}. We can apply integration by parts to the Riemann-Stieltjes integral $J$ to get that 
\begin{align*}
    J^{n,\ell}(s,t)=&\left(\cos(b_1^n \pi t)-\cos(b_1^n \pi s)\right)\cos(b_2^\ell \pi t)-\left(\cos(b_1^n \pi s)-\cos(b_1^n \pi s)\right)\cos(b_2^\ell \pi s)\\
    &-\int_s^t \cos(b_2^\ell \pi r)\, d\cos(b_1^n \pi r)\\
    =&\left(\cos(b_1^n \pi t)-\cos(b_1^n \pi s)\right)\cos(b_2^\ell \pi t)-\int_s^t \cos(b_2^\ell \pi r)\, d\cos(b_1^n \pi r)\\
    =&\left(\cos(b_1^n \pi t)-\cos(b_1^n \pi s)\right)\cos(b_2^\ell \pi t){+}b_1^n \pi \int_s^t \cos(b_2^\ell \pi r) \sin(b_1^n \pi r) dr.
\end{align*}
By the triangle inequality, Lipschitz property of {cosine} and the bounds $|\cos(x)|\leq 1,|\sin(x)|\leq 1$ we have that
\begin{align*}
     |J^{n,\ell}(s,t)|&\lesssim b_1^n (t-s).
\end{align*}

\textbf{Bound (iv)}. We will split this into three cases. First, when $b_1^n=b_2^\ell$. Second, when $b_1^n\neq b_2^\ell$ and $\log_{b_1}\!(b_2)$ is irrational. Third, when $b_1^n\neq b_2^\ell$ but $\log_{b_1}\!(b_2)$ is rational.

\textbf{Case 1 ($b_1^n=b_2^\ell$).} If $b_1^n=b_2^\ell$ then 
\begin{equation}
    J^{n,\ell}(s,t)=\frac{1}{2}(\cos(b_1^n \pi t)-\cos(b_1^n \pi s))^2=\frac{1}{2}(\cos(b_2^\ell \pi t)-\cos(b_2^\ell \pi s))^2.
\end{equation}
Therefore 
\begin{equation}
    |J^{n,\ell}(s,t)|\lesssim 1.
\end{equation}
For $b_1^n\neq b_2^\ell$ we can evaluate the integral as
{
\begin{align*}
    J^{n,\ell}(s,t)&=-b_2^\ell \pi \left(\frac{\cos(t(b_1^n-b_2^\ell))}{2(b_1^n-b_2^\ell)}-\frac{\cos(s(b_1^n-b_2^\ell))}{2(b_1^n-b_2^\ell)}-\frac{\cos(t(b_1^n+b_2^\ell))}{2(b_1^n+b_2^\ell)}+\frac{\cos(s(b_1^n+b_2^\ell))}{2(b_1^n+b_2^\ell)}\right)\\
    &`-\cos(b_1^n\pi s)(\cos(b_2^\ell \pi t)-\cos(b_2^\ell \pi s)),
\end{align*}
using the indefinite integral formula
\begin{equation*}
    \int \cos(mr)\sin(nr) dr=\frac{\cos(r(m-n))}{2(m-n)}-\frac{\cos(r(m+n))}{2(m+n)}+C.
\end{equation*}
Therefore we have that
\begin{equation*}
    |J^{n,\ell}(s,t)|\lesssim \frac{b_2^\ell}{|b_2^\ell-b_1^n|}+1.
\end{equation*}
}

\textbf{Case 2 ($b_1, b_2$ multiplicatively independent).} Now let $\log_{b_1}\!(b_2)$ be irrational. Then, using Lemma \ref{lem:Feldman} we have that there exists some constant $P>0$ depending only on $b_1,b_2$ so that
{\begin{align*}
     |J^{n,\ell}(s,t)|&\lesssim \frac{b_2^{
     \ell}}{\max\{b_1^n,b_2^\ell\}}(\max\{3,n,\ell\})^P+1\\
     &\lesssim n^P \ell^P. 
\end{align*}}
{Recall that} $0<\varepsilon<\min\{\alpha_1,\alpha_2\}$. There exists some constants $C^1(\varepsilon),C^2(\varepsilon)>0$ depending on $P$, 
$\varepsilon$ and on $b_1,b_2$ so that $n^P\leq C^1(\varepsilon) b_1^{\varepsilon n}$ for all $n$ and $\ell^P\leq C^2(\varepsilon) b_2^{\varepsilon \ell}$ for all $\ell$. Therefore we have that bound
\begin{equation}
    |J^{n,\ell}(s,t)|\lesssim b_1^{\varepsilon n}b_2^{\varepsilon \ell}.
\end{equation}

\textbf{Case 3 ($b_1, b_2$ multiplicatively dependent).} Finally, let $\log_{b_1}(b_2)$ be rational. Then there exists some positive integers $b, q_1,q_2$ so that $b_1=b^{q_1}$ and $b_2=b^{q_2}$. Then without loss of generality letting $\ell q_1>n q_2$ we have that
{
\begin{align*}
    |J^{n,\ell}(s,t)|&\lesssim \frac{b_2^\ell}{|b_2^{\ell}-b_1^{n}|}+1\\
    &=\frac{b^{\ell q_1}}{b^{\ell q_1}-b^{n q_2}}+1\\
    &=\frac{b^{\ell q_1-nq_2}}{b^{\ell q_1-nq_2}-1}+1\\
    &\leq \frac{b^{\ell q_1-nq_2}}{b^{\ell q_1-nq_2}-\frac12b^{\ell q_1-nq_2}}+1\\
    &\lesssim 1.
\end{align*}}

Putting cases 1,2,3 together yields bound \textbf{(iv)}
\begin{equation}\label{eq:bound-iv}
    |J^{n,\ell}(s,t)|\lesssim b_1^{\varepsilon n} b_2^{\varepsilon \ell}.
\end{equation}

Using these four bounds we can show the H\"older regularity. Similarly to the proof of Proposition \ref{prop:Holder-continuity}, let $N_1, N_2$ be two nonnegative integers so that 
\begin{align*}
    b_1^{-(N_1+1) }&< |t-s|\leq b_1^{-N_1}\\
    b_2^{-(N_2+1) }&< |t-s|\leq b_2^{-N_2}.
\end{align*}
Without loss of generality, we may assume $N>N_1,N_2$. Indeed, if $N\leq N_1$ or $N\leq N_2$ then 
\begin{align*}
    |I^N(s,t)|&\leq \sum_{n=0}^{N}\sum_{\ell=0}^{N}a_1^n a_2^\ell |J^{n,\ell}(s,t)|\\
    &\leq \sum_{n=0}^{\max\{N_1,N\}}\sum_{\ell=0}^{\max\{N_2,N\}}a_1^n a_2^\ell |J^{n,\ell}(s,t)|
\end{align*}
and the  argument below will bound $I^N$ as well. With this assumption in hand, we can split $I^N$ as
\begin{align*}
I^N(s,t)=&\sum_{n=0}^{N_1}\sum_{\ell=0}^{N_2}a_1^n a_2^\ell J^{n,\ell}(s,t)+\sum_{n=N_1+1}^{N}\sum_{\ell=0}^{N_2}a_1^n a_2^\ell J^{n,\ell}(s,t)\\
&+\sum_{n=0}^{N_1}\sum_{\ell=N_2+1}^{N}a_1^n a_2^\ell J^{n,\ell}(s,t)+\sum_{n=N_1+1}^{N}\sum_{\ell=N_2+1}^{N}a_1^n a_2^\ell J^{n,\ell}(s,t)\\
=:&\:\mathbf{I}+\mathbf{II}+\mathbf{III}+\mathbf{IV}.
\end{align*}
Applying bound \textbf{(i)} to \textbf{I}, bound \textbf{(ii)} to \textbf{II}, bound \textbf{(iii)} to \textbf{III}, and bound \textbf{(iv)} to \textbf{IV} yields that 
\begin{align*}
    |\mathbf{I}|&\lesssim \sum_{n=0}^{N_1}\sum_{\ell=0}^{N_2}a_1^n a_2^\ell b_1^n b_2^\ell (t-s)^2=\left(\sum_{n=0}^{N_1} a_1^n b_1^n (t-s)\right)\left(\sum_{\ell=0}^{N_2} a_2^\ell b_2^\ell(t-s)\right),\\
    |\mathbf{II}|&\lesssim \sum_{n=N_1+1}^N\sum_{\ell=0}^{N_2}a_1^n a_2^\ell b_2^\ell (t-s)=\left(\sum_{n=N_1+1}^N a_1^n\right)\left(\sum_{\ell=0}^{N_2} a_2^\ell b_2^\ell (t-s)\right),\\
    |\mathbf{III}|&\lesssim \sum_{n=0}^{N_1}\sum_{\ell=N_2+1}^{N}a_1^n a_2^\ell  b_1^n (t-s)=\left(\sum_{n=0}^{N_1}a_1^n b_1^n(t-s)\right)\left(\sum_{\ell=n_2+1}^{N} a_2^\ell\right),\\
    |\mathbf{IV}|&\lesssim \sum_{n=N_1+1}^{N}\sum_{\ell=N_2+1}^{N}a_1^n a_2^\ell b_1^{\varepsilon n} b_2^{\varepsilon \ell}=\left(\sum_{n=N_1+1}^{N}a_1^n b_1^{\varepsilon n}\right)\left(\sum_{\ell=N_2+1}^{N}a_2^\ell b_2^{\varepsilon \ell} \right).
\end{align*}
By the choice of $N_1,N_2$ and standard bounds for geometric series we obtain the following estimates:
\begin{align*}
  &\sum_{n=0}^{N_1} a_1^{n}b_1^n (t-s)\lesssim a_1^{N_1}b_1^{N_1}(t-s)\leq a_1^{N_1}= b_1^{-\alpha_1 N_1}\lesssim |t-s|^{\alpha_1},\\
  &\sum_{\ell=0}^{N_2} a_2^{\ell}b_2^\ell (t-s)\lesssim |t-s|^{\alpha_2},\\
  &\sum_{n=N_1+1}^N a_1^n\leq \sum_{n=N_1+1}^\infty a_1^n\lesssim a_1^{N_1}\lesssim |t-s|^{\alpha_1}\\
 & \sum_{\ell=N_2+1}^N a_2^\ell \lesssim |t-s|^{\alpha_2},\\
 &\sum_{n=N_1+1}^{N}{a_1^n} b_1^{\varepsilon n}\lesssim a_1^{N_1} b_1^{\varepsilon N_1}\lesssim |t-s|^{\alpha_1-\varepsilon},\\
 & \sum_{\ell=N_2+1}^{N}a_2^\ell b_2^{\varepsilon \ell}\lesssim |t-s|^{\alpha_2-\varepsilon}.
\end{align*}
  Collecting these six bounds proves the relation \eqref{eq:bound-on-I-1}.
\end{proof}

\begin{corollary}\label{corollary:existence-of-I}
For all $(s,t)$ with $0\leq s\leq t\leq 1$, the limit
\begin{equation}
    \lim_{N\to \infty} I_{(b_1,\alpha_1),(b_2,\alpha_2)}^N(s,t)=:I_{(b_1,\alpha_1),(b_2,\alpha_2)}(s,t)
\end{equation}
exists. 
\end{corollary}
\begin{proof}
Let $\varepsilon>0$ be such that $\varepsilon<\min\{\alpha_1, \alpha_2\}$. Applying bound \textbf{(iv)} to $I^N$ shows that
    \begin{align*}
        |I^N(s,t)|&\lesssim \sum_{n=0}^N \sum_{\ell=0}^Na_1^n a_2^\ell b_1^{\varepsilon n} b_2^{\varepsilon \ell}\\
        &=\sum_{n=0}^N \sum_{\ell=0}^N b_1^{-n\alpha_1} b_2^{-\ell \alpha_2} b_1^{\varepsilon n} b_2^{\varepsilon \ell}.
    \end{align*}
Absolute convergence follows. 
\end{proof}

With both the pointwise convergence and the uniform estimate \eqref{eq:bound-on-I-1} we can show convergence in the rough topology and even give a geometric rate of convergence. 

\begin{theorem}\label{theorem:main-full}
Let $t\mapsto W(t)=(W_{b_1,\alpha_1}(t),...,W_{b_d,\alpha_d}(t))$ be a $d$-dimensional Weierstrass function where $\alpha_1,...,\alpha_d>1/3$. Let $\alpha=\min\{\alpha_1,...,\alpha_d\}$.

For every $N\in\mathbb N$,  consider the vector-valued function $W_N:\Delta_2^{(0,1)}\to\mathbb{R}^d$ by setting $
W_N(t):=\left(W_{b_1,\alpha_1,N}(t),...,W_{b_d,\alpha_d,N}(t)\right)$. Similarly,  consider the matrix-valued function $\mathbb{A}_N:\Delta_2^{(0,1)}\to \mathbb R^{d\times d}$, defined  entrywise by means of the iterated integrals  \eqref{I^N_(b1,alpha1)(b2,alpha2)}  by setting $\mathbb{A}_N^{i,j}(s,t):=I^N_{(b_i,\alpha_i),(b_j,\alpha_j)}(s,t)$ for $1\leq i,j\leq d$. Let $\mathbf{W}_N:\Delta_2^{(0,1)}\to \mathbb{R}^d\oplus\mathbb{R}^{d\times d}$, $\mathbf{W}_N(s,t)=\left(W_N(t)-W_N(s),\mathbb{A}_N(s,t)\right)$.
Finally, define $\mathbb{A}:\Delta_2^{(0,1)}\to \mathbb R^{d\times d}$ entrywise as $\mathbb{A}^{i,j}(s,t):=\displaystyle\lim_{N\to\infty}\mathbb{A}_N^{i,j}(s,t)$ and let $\mathbf{W}:\Delta_2^{(0,1)}\to \mathbb{R}^d\oplus\mathbb{R}^{d\times d}$, $\mathbf{W}(s,t)=\left(W(t)-W(s),\mathbb{A}(s,t)\right)$. 
Then for every $\varepsilon\in(0,\alpha)$ we have that $\mathbf{W}\in\mathscr C_g^{\alpha-\varepsilon}$. 

{Furthermore, for every $\beta\in (\alpha-\varepsilon,\alpha)$ let us set  $\kappa:=1-\frac{\alpha-\varepsilon}{\beta}\in (0,1)$. Then for every $\varepsilon'\in (0,\alpha)$, setting $\rho=\left(\max\!\left\{b_1^{-\alpha_1+\varepsilon'},\ldots ,b_d^{-\alpha_d+\varepsilon'}\right\}\right)^{\kappa}\in(0,1)$,  
we have the estimate
\begin{equation}\label{eq:geometric-convergence}
    \|\mathbf{W}-\mathbf{W}_N\|_{\alpha-\varepsilon,2\alpha-2\varepsilon}\lesssim \rho^{N}.
\end{equation}
In other words, $\mathbf{W}_N$ converges to $\mathbf{W}$ geometrically fast in the rough topology as $N\to\infty$. }
\end{theorem}

\begin{proof} {We may use the interpolation inequality (see \cite{Friz-Victoir-Book}, Proposition 5.5)
\begin{equation*}
    \|\mathbf{W}-\mathbf{W}_N\|_{\alpha-\varepsilon,2\alpha-2\varepsilon}\leq \|\mathbf{W}-\mathbf{W}_N\|_{\beta,2\beta}^{1-\kappa} \|\mathbf{W}-\mathbf{W}_N\|_\infty^{\kappa}.
\end{equation*}
 Thanks to the uniform bound on the $\beta$ rough-path distance given in Proposition \ref{prop:uniform-bound-Levy-area} we just need to bound the supremum norm of $\mathbf{W}-\mathbf{W}_N$. For the Weierstrass function $W_{b_i,\alpha_i}$ we have the estimate
 \begin{equation}\label{eq:sup-first}
     \sup_{s,t\in [0,1]} \left|\sum_{n=N+1}^\infty b_i^{-n\alpha_i} (\cos(b_i^n \pi t)-\cos(b_i^n\pi s))\right|\leq \sup_{s,t\in [0,1]} 2\sum_{n=N+1}^\infty b_i^{-n\alpha_i} \lesssim b_i^{-N\alpha_i}. 
 \end{equation}
 For the iterated integral $I^N_{(b_i,\alpha_i),(b_j,\alpha_j)}$ we need to estimate
 \begin{equation*}
      \sup_{s,t\in [0,1]}\left|\sum_{n=1}^\infty \sum_{\ell=N+1}^\infty b_i^{-n\alpha_i}b_j^{-\ell\alpha_j} J^{n,\ell}(s,t)+\sum_{n=N+1}^\infty \sum_{\ell=1}^\infty b_i^{-n\alpha_i}b_j^{-\ell\alpha_j} J^{n,\ell}(s,t)+\sum_{n=N+1}^\infty \sum_{\ell=N+1}^\infty b_i^{-n\alpha_i}b_j^{-\ell\alpha_j} J^{n,\ell}(s,t)\right|.
 \end{equation*} 
 The triangle inequality, bound \textbf{(iv)} given in \eqref{eq:bound-iv} and the standard tail estimates for geometric series give the estimate
 \begin{equation}\label{eq:sup-second}
    \sup_{s,t\in [0,1]}|I_{(b_i,\alpha_i),(b_j,\alpha_j)}(s,t)-I^N_{(b_i,\alpha_i),(b_j,\alpha_j)}(s,t)|\lesssim b_i^{(-\alpha_i+\varepsilon')N} b_j^{(-\alpha_j+\varepsilon')N}
 \end{equation} 
 for any $\varepsilon'$ with $\min\{\alpha_i,\alpha_j\}>\varepsilon'>0$. Using estimates \eqref{eq:sup-first} and \eqref{eq:sup-second} along with the bound \eqref{eq:bound-on-I-1} on the norm $\|\mathbf{W}-\mathbf{W}_N\|_{\beta,2\beta}$ yields the claimed relation \eqref{eq:geometric-convergence}.}
\end{proof}
\begin{remark}\label{remark:sin-cos}
    Note that we prove everything for Weierstrass functions with $\cos$. However, the same arguments work with little difference if we replace every {cosine} with {sine}. However, as shown in \cite{Imkeller-Weierstrass} and further studied in \cite{Imkeller-Weierstrass-2}, if we replace one  {cosine} with  {sine}, the arguments above no longer work.
\end{remark}
{Theorem \ref{theorem:main-full} is optimal in the sense that one cannot remove the $-\varepsilon$ factor. Indeed, the partial sums of the first order process do not even converge in the function space $C^\alpha([0,1],\mathbb R^d)$. We include a proof of this fact for expository reasons, although we suspect it is well known. 
\begin{proposition}\label{proposition:ref-1}
    The partial sums $W_{b,\alpha,N}$ do not converge in $C^\alpha([0,1],\mathbb R^d)$ to $W_{b,\alpha}$.
\end{proposition}
\begin{proof}
    First, let $b$ be odd. In this case by selecting $t=b^{-N}$ and $s=0$ in the definition of the H\"older norm we get
\begin{align*}
\|W_{b,\alpha}-W_{b,\alpha,N}\|_\alpha&:=\sup_{t\neq s\in [0,1]} \frac{|(W_{b,\alpha}-W_{b,\alpha,N})(t)-(W_{b,\alpha}-W_{b,\alpha,N})(s)|}{|t-s|^\alpha}\\
&\geq \frac{|(W_{b,\alpha}-W_{b,\alpha,N})(b^{-N})-(W_{b,\alpha}-W_{b,\alpha,N})(0)|}{|b^{-N}|^\alpha}.
\end{align*}
Using the definition of $W_{b,\alpha}$ and $W_{b,\alpha,N}$ yields
\begin{align*}
|(W_{b,\alpha}-W_{b,\alpha,N})(b^{-N})-(W_{b,\alpha}-W_{b,\alpha,N})(0)||&=\left|\sum_{n=N+1}^\infty b^{-\alpha n}(\cos(b^{n-N} \pi)-\cos(0))\right|\\
&=\left|\sum_{n=N+1}^\infty b^{-\alpha n}(-1-1)\right|\\
&=\left|\sum_{n=1}^\infty b^{-\alpha(n-N)}2\right|\\
&=b^{-N\alpha} \frac{2b^{-\alpha}}{1-b^{-\alpha}},
\end{align*}
where we used that $b$ was odd in the second line. Therefore
\begin{align*}
    \|W_{b,\alpha}-W_{b,\alpha,N}\|_\alpha&\geq b^{-N\alpha} \frac{2b^{-\alpha}}{1-b^{-\alpha}} b^{N\alpha}\\
    &=\frac{2b^{-\alpha}}{1-b^{-\alpha}}\\
    &\not\to 0.
\end{align*}
Now, let $b$ be even. In this case we can let $t=b^{-N-1}$ and $s=0$ to get
\begin{align*}
\|W_{b,\alpha}-W_{b,\alpha,N}\|_\alpha&:=\sup_{t\neq s\in [0,1]} \frac{|(W_{b,\alpha}-W_{b,\alpha,N})(t)-(W_{b,\alpha}-W_{b,\alpha,N})(s)|}{|t-s|^\alpha}\\
&\geq \frac{|(W_{b,\alpha}-W_{b,\alpha,N})(b^{-N-1})-(W_{b,\alpha}-W_{b,\alpha,N})(0)|}{|b^{-N-1}|^\alpha}.
\end{align*}
Again using the definition of $W_{b,\alpha}$ and $W_{b,\alpha,N}$ yields
\begin{align*}
|(W_{b,\alpha}-W_{b,\alpha,N})(b^{-N-1})-(W_{b,\alpha}-W_{b,\alpha,N})(0)|&=\left|\sum_{n=N+1}^\infty b^{-\alpha n}(\cos(b^{n-N-1} \pi)-\cos(0))\right|\\
&=2 b^{-\alpha(N+1)}
\end{align*}
where we used that $\cos(b^{n-N-1}\pi)-\cos(0)=0$ for all $n>N+1$ (thanks to $b$ being even) and is $-2$ if $n=N+1$. Therefore
\begin{align*}
    \|W_{b,\alpha}-W_{b,\alpha,N}\|_\alpha&\geq 2 b^{-\alpha(N+1)}b^{\alpha(N+1)}\\
    &=2\\
    &\not\to 0.
\end{align*}
\end{proof}}
{\subsection{Pointwise Convergence}
The power of our main Theorem \ref{theorem:main-full} is the convergence in the rough topology of the approximating iterated integrals. We used the estimate \eqref{eq:bound-on-I-1} in order to get convergence in the rough topology and we needed an estimate on the supremum norm to get the rate of convergence. However, if we are interested only in pointwise convergence, the following proposition addresses a broad class of functions defined as trigonometric series. 
\begin{proposition}\label{prop:ref-2}
    Let $\{\alpha_n\}_{n\in \mathbb N}$ and $\{\beta_n\}_{n\in \mathbb N}$ be two sequences of positive real numbers. Define
    \begin{equation*}
        f(t)=\sum_{n=1}^\infty c_n e^{i\alpha_n t},\qquad g(t)=\sum_{n=1}^\infty d_n e^{i\beta_n t},
    \end{equation*}
    where $\{c_n\}_{n\in \mathbb N}$ and $\{d_n\}_{n\in \mathbb N}$ are two sequences of complex numbers so that
    \begin{equation*}
        \sum_{n=1}^\infty |c_n|<\infty, \qquad \sum_{n=1}^\infty |d_n|<\infty.
    \end{equation*}
    Define, for $N\in \mathbb N$, the partial sums
    \begin{equation*}
        f_N(t)=\sum_{n=1}^N c_n e^{i\alpha_n t},\qquad g_N(t)=\sum_{n=1}^N d_n e^{i\beta_n t}.
    \end{equation*}
    Then for every $s,t\in \mathbb R$ the limit
    \begin{equation*}
        \lim_{N\to\infty} \int_s^t f_N(r) dg_N(r)
    \end{equation*}
    exists. 
\end{proposition}
\begin{proof}
    As $g_N$ is smooth we can write
    \begin{equation*}
         \int_s^t f_N(r) dg_N(r)= \int_s^t f_N(r) g_N'(r)dr
    \end{equation*}
    and by linearity 
    \begin{equation*}
        g_N'(r)=\sum_{n=1}^N (i\beta_n) d_n e^{i\beta_n r}.
    \end{equation*}
    For $N>M$ we can write
    \begin{equation*}
        f_N(r)g_N'(r)-f_M(r)g_m'(r)=f_N(r)\left(g_N'(r)-g_M'(r)\right)+g_M'(r)\left(f_N(r)-f_M(r)\right).
    \end{equation*}
    We have that
    \begin{align*}
        \int_s^t f_N(r)(g_N'(r)-g_M'(r))dr&=\int_s^t \sum_{n_1=1}^N\sum_{n_2=M+1}^N c_{n_1} (i\beta_{n_2})d_{n_2} e^{ir(\alpha_{n_1}+\beta_{n_2})}dr\\
        &=\sum_{n_1=1}^N\sum_{n_2=M+1}^N c_{n_1} (i\beta_{n_2})d_{n_2} \frac{e^{it(\alpha_{n_1}+\beta_{n_2})}-e^{is(\alpha_{n_1}+\beta_{n_2})}}{i(\alpha_{n_1}+\beta_{n_2})}.
    \end{align*}
    Crucially, as $\alpha_n$ and $\beta_n$ are positive, we get that
    \begin{equation*}
        \left|(i\beta_{n_2}) \frac{e^{it(\alpha_{n_1}+\beta_{n_2})}-e^{is(\alpha_{n_1}+\beta_{n_2})}}{i(\alpha_{n_1}+\beta_{n_2})}\right|\leq 2.
    \end{equation*}
    Using this bound we obtain the estimate
    \begin{equation*}
        \left| \int_s^t f_N(r)(g_N'(r)-g_M'(r))dr\right|\leq 2 \sum_{n_1=1}^N |c_{n_1}|\sum_{n_2=M+1}^N|d_{n_2}|.
    \end{equation*}
    As $M,N\to\infty$ the right hand side of this inequality goes to $0$. Analogously we get that
     \begin{equation*}
        \left| \int_s^t g_M'(r)(f_N(r)-f_M(r))dr\right|\leq 2 \sum_{n_1=M+1}^N |c_{n_1}|\sum_{n_2=1}^N|d_{n_2}|,
    \end{equation*}
    with again the right hand side going to $0$ as $M,N\to\infty$.
\end{proof}
One important consequence of Proposition \ref{prop:ref-2} is to consider the ``complex" Weierstrass functions
\begin{equation*}
    f(t)=\sum_{n=1}^\infty b_1^{-\alpha_1 n}e^{ib_1^n \pi t}
\end{equation*}
and 
\begin{equation*}
    g(t)=\sum_{n=1}^\infty b_2^{-\alpha_2 n}e^{ib_2^n \pi t}.
\end{equation*}
In this case, Proposition \ref{prop:ref-2} implies that the limit
\begin{equation*}
    \lim_{N\to\infty}\int_s^t f_N(r)dg_N(r)
\end{equation*}
exists. Define
\begin{equation*}
    F_1(t)=\sum_{n=1}^\infty b_1^{-\alpha_1 n}\cos(b_1^n \pi t),\qquad F_2(t)=\sum_{n=1}^\infty b_1^{-\alpha_1 n}\sin(b_1^n \pi t)
\end{equation*}
and 
\begin{equation*}
    G_1(t)=\sum_{n=1}^\infty b_2^{-\alpha_2 n}\cos(b_2^n \pi t),\qquad G_2(t)=\sum_{n=1}^\infty b_2^{-\alpha_2 n}\sin(b_2^n \pi t)
\end{equation*}
so that $f=F_1+F_2 i$ and $g=G_1+G_2 i$. Proposition \ref{prop:ref-2} along with taking the imaginary part implies that one can make sense of the integral 
\begin{equation*}
    \int_s^t F_1 (r) dG_2(r)+\int_s^t F_2(r) dG_1(r).
\end{equation*}
 However as was shown in \cite{Imkeller-Weierstrass} one cannot make sense of the integral
\begin{equation*}
    \int_s^t F_1 (r) dG_2(r)-\int_s^t F_2(r) dG_1(r)
\end{equation*}
as a limit of integrals of partial sums. Our Theorem \ref{theorem:main-full} along with Remark \ref{remark:sin-cos} show that one can make sense of
\begin{equation*}
     \int_s^t F_1 (r) dG_1(r)
\end{equation*}
and 
\begin{equation*}
     \int_s^t F_2 (r) dG_2(r).
\end{equation*}
}

\section{An Example}
One of the benefits of rough paths theory is a robust approximation theory. By standard approximation theory for rough differential equations (see \cite{Friz-Hairer-Book}, Chapter 8), we can approximate solutions to rough differential equations (RDEs) driven by a {vector-valued} Weierstrass function by a RDE/ODE driven by the smooth truncated sums $W_N$. 
Let $d=2$ and 
\begin{align}\label{ex:2x2-matrix}
    M(Y)=\left(\begin{array}{cc}
    0 & Y_2/3 \\
    Y_1/2 & 0 
\end{array}\right).
\end{align}

Figure \ref{fig:three-SDE-approximations} shows the solutions $Y_N=(Y_{1,N},Y_{2,N})$ of the  ordinary differential equation  $dY=M(Y)dW_N$ driven by $W_N=(W_{b_1,\alpha_1,N},W_{b_2,\alpha_2,N})$ where 
 $b_1,b_2, \alpha_1, \alpha_2$ are as in Figure \ref{fig:two-Weierstrass-functions}.
 Due to standard approximation theory, $Y_N$ converges as $N\to\infty$ to the solution of the rough differential equation 
 \begin{align}\label{RDE}
    dY=M(Y)d\mathbf{W},
 \end{align}
where $\mathbf{W}=(W,\mathbb{A})$ is as in Theorem \ref{theorem:main-full}. 
\begin{figure}[h!]
    \centering
\includegraphics[width=12.7cm]{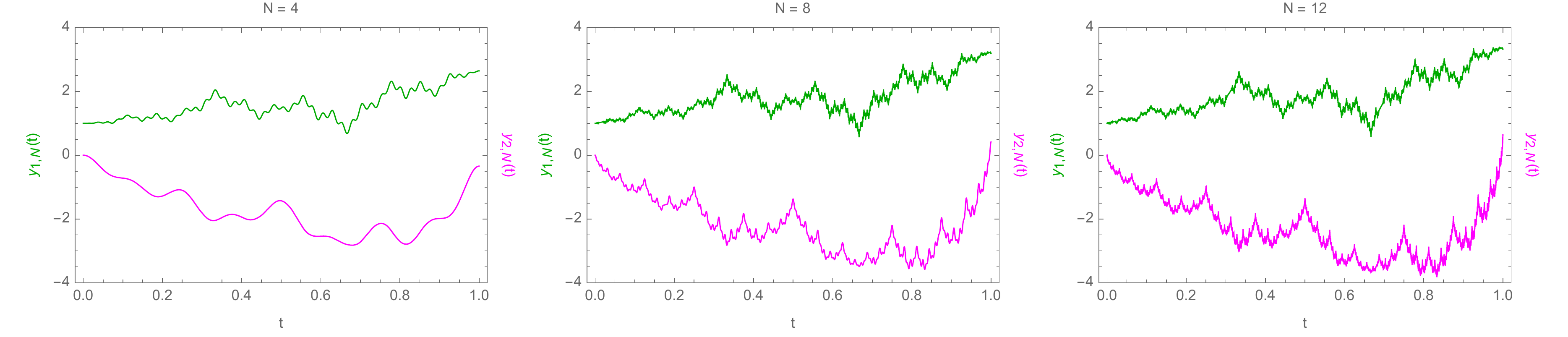}
    \caption{The approximate solutions to \eqref{RDE} with $M$ as in \eqref{ex:2x2-matrix}, and initial condition $Y_1(0)=1$, $Y_2(0)=0$, where $N=4,8,12$ and $b_1=2, \alpha_1=-\frac{\ln 18/25}{\ln 2}$ and $b_2=3, \alpha_2=-\frac{\ln 3/5}{\ln 3}$ as in Figure \ref{fig:two-Weierstrass-functions}.}
    \label{fig:three-SDE-approximations}
\end{figure}

\bibliographystyle{crplain}
\nocite{*}

\bibliography{bibliography}
\end{document}